\documentclass[12pt, a4paper, twosided, reqno]{amsart}
\usepackage{enumerate, cite}
\usepackage{hyperref}

\numberwithin{equation}{section}
\newtheorem{theorem}{Theorem}[section]

\theoremstyle{remark}
\newtheorem{remark}[theorem]{Remark}
\newtheorem{example}[theorem]{Example}
\newtheorem{definition}[theorem]{Definition}

\makeatletter
\@namedef{subjclassname@2020}{%
\textup{2020} Mathematics Subject Classification}
\makeatother
\begin{document}

\title[study of Squeezing function Corresponding to Polydisk]{Further study of Squeezing function Corresponding to Polydisk}

\author[N. Gupta]{Naveen Gupta}
\address{Department of Mathematics, University of Delhi,
Delhi--110 007, India}
\email{ssguptanaveen@gmail.com}

\author[A. Kumar]{Akhil Kumar}

\address{Department of Mathematics, University of Delhi,
	Delhi--110 078, India}
\email{98.kumarakhil@gmail.com}

\begin{abstract}
In the present article, we investigate some properties of squeezing function corresponding to polydisk. We give some explicit expressions of squeezing function 
corresponding to polydisk.
\end{abstract}
\keywords{squeezing function; extremal map; holomorphic homogeneous regular domain.}
\subjclass[2020]{32F45, 32H02}
\maketitle

%\vspace{-0.6cm}
\section{Introduction}
In 2012, F. Deng, Q. Guan, L. Zhang in their paper \cite{2012}, introduced the notion of squeezing function by following the work of K. Liu, X. Sun,
 S. T. Yau, \cite{Yau2004}, \cite{Yau2005} and  S. K. Yeung \cite{yeung}. 
\begin{definition}
Let $\Omega\subset \mathbb{C}^n$ be a bounded domain. For $z\in{\Omega}$ and a holomorphic embedding $f:\Omega\to \mathbb{B}^n$  with $f(z)=0$, define 
$$S_{\Omega}(f,z):=\sup \{r:\mathbb{B}^n(0,r)\subseteq f(\Omega)\},$$
where $\mathbb{B}^n$ denotes unit ball in $\mathbb{C}^n$ and $\mathbb{B}^n(0,r)$ denotes the ball centered at origin with radius $r$ in $\mathbb{C}^n$.
The squeezing function on $\Omega$, denoted by $S_{\Omega}$
is defined as
$$S_{\Omega}(z):=\sup_f\{r:\mathbb{B}^n(0,r)\subseteq f(\Omega)\},$$ 
where supremum is taken over all holomorphic embeddings $f:\Omega\to \mathbb{B}^n$ with $f(z)=0.$
\end{definition}
For other important properties of squeezing functions  
corresponding to unit ball one can refer the following papers:
\cite{uniform-squeezing}, \cite{kim-joo}, \cite{andreev}, \cite{zimmer2}, \cite{kaushalprachi}, \cite{recent}.   
Recently  \cite{polydisc}  N. Gupta, S. K. Pant introduced squeezing function corresponding to polydisk. 
\begin{definition}  \cite{polydisc}
Let $\Omega\subset \mathbb{C}^n$ be a bounded domain. For $z\in{\Omega}$ and holomorphic embedding $f:\Omega\to \mathbb{D}^n$  with $f(z)=0$, define 
$$T_{\Omega}(f,z):=\sup \{r:\mathbb{D}^n(0,r)\subseteq f(\Omega)\},$$
where $\mathbb{D}^n$ denotes unit polydisk and $\mathbb{D}^n(0,r)$ denotes the polydisk centered at origin with radius $r$ in $\mathbb{C}^n$.
The squeezing function corresponding to polydisk on $\Omega$, denoted by $T_{\Omega}$,
is defined as
$$T_{\Omega}(z):=\sup_f\{r:\mathbb{D}^n(0,r)\subseteq f(\Omega)\},$$ 
where supremum is taken over all holomorphic embeddings $f:\Omega\to \mathbb{D}^n$ with $f(z)=0.$
\end{definition}

From definition, it is clear that for any bounded domain $\Omega$, $T_{\Omega}$ is invariant under biholomorphic transformations, that is, 
$T_{f(\Omega)}(f(z))=T_{\Omega}(z)$, where $f$ is any biholomorphic transformation of $\Omega$. 
Also by \cite[Lemma~1.2]{polydisc}, holomorphic homogeneous regular(HHR) domains are same in both the settings, that is,
 if squeezing function corresponding to polydisk $T_{\Omega}$ on $\Omega$ admits positive lower bound, then $\Omega$ is HHR.

 Recently Rong and Yang \cite{rong} extended this notion to generalised squeezing function by taking the embeddings in a bounded, balanced, convex domain. 
 The object, $T_D$(for a bounded domain $D$) is one particular example of generalised squeezing function.
  N. Gupta, S. K. Pant in \cite{d-balanced} introduced more general notion $d$-balanced squeezing function by taking the embeddings in 
  bounded, $d$-balanced domain, convex domains. Now the question comes, what is the importance of the object $T_{D}$, when there are more general set ups available. 
  It seems to us that the set up involving polydisk provides a concrete model space to work with(note that $T_D$ is an example of a generalised squeezing function). 
  The better perspective is gained to deal with difficulty arising in
  arbitrary balanced domains. 
 \section{Main theorems}
 
 We begin with the following observation: In \cite{polydisc}, there is a statement which says that for bounded homogeneous domain $\Omega\subset\mathbb{C}^{n}$, 
  either $S_{\Omega}(z)= \dfrac{1}{\sqrt{n}}T_{\Omega}(z)$ or $T_{\Omega}(z)= \dfrac{1}{\sqrt{n}}S_{\Omega}(z)$. 
  We see that this statement is not true in general. To see this consider the following remark. 
\begin{remark} 
Let $\Omega = \mathbb{B}^{n}\times\mathbb{B}^{n}\times\ldots\times\mathbb{B}^{n}\subset \mathbb{C}^{n^{2}}$ be a bounded domain.
 Then for $z=(z_{1}, z_{2},\ldots, z_{n})\in \Omega$ and $n > 1$(here $z_i=(z_{i1}, z_{i2},\ldots, z_{in})$), we have
$$S_{\Omega}(z)\neq \dfrac{1}{n}T_{\Omega}(z)\  \ \ \mbox{and}\  \ \ T_{\Omega}(z)\neq\dfrac{1}{n} S_{\Omega}(z).$$
Since $\mathbb{B}^{n}$ is a classical symmetric domain, so by \cite[Theorem~7.5]{2012}
$$S_{\Omega}(z) = (S_{\mathbb{B}^{n}}(z_{1})^{-2}+S_{\mathbb{B}^{n}}(z_{2})^{-2}+\ldots +S_{\mathbb{B}^{n}}(z_{n})^{-2})^{-1/2} = \dfrac{1}{\sqrt{n}}.$$
Also, by \cite[Proposition~4.6]{polydisc}, $T_{\Omega}(z)\geq\displaystyle\min_{1\leq i \leq n}T_{\mathbb{B}^{n}}(z_{i})$ and since $T_{\mathbb{B}^{n}}(z_{i})\equiv \dfrac{1}{\sqrt{n}}$ for each $i$, this implies $T_{\Omega}(z)\geq \dfrac{1}{\sqrt{n}}$.\\ 
If $S_{\Omega}(z)= \dfrac{1}{n}T_{\Omega}(z)$, then $T_{\Omega}(z)=\sqrt{n}$ which is not possible.\\
If $T_{\Omega}(z)= \dfrac{1}{n}S_{\Omega}(z)$, then $T_{\Omega}(z)=\dfrac{1}{n^{3/2}} < \dfrac{1}{\sqrt{n}}$ which is not possible. 
It should be noted that $\Omega = \mathbb{B}^{n}\times\mathbb{B}^{n}\times\ldots\times\mathbb{B}^{n}\subset \mathbb{C}^{n^{2}}$ is bounded homogeneous domain.
\end{remark}
Before stating our next theorem we would like to fix notation for Poincaré metric on unit disc. We consider positive real valued function  $\sigma:[0, 1)\to \mathbb{R^{+}}$ defined as
 $\sigma(x) = \log\dfrac{1+x}{1-x}$. It is clear that $\sigma$ is strictly increasing function and for $z\in\mathbb{D}$, $\sigma(|z|)$ is the Poincaré distance between 0 and $z$. One can observe that the inverse of this function can be expressed as $\tanh\left(\dfrac{x}{2}\right)$.

\begin{theorem}
Let $\Omega = \Omega_{1}\times\Omega_{2}\times\ldots\times\Omega_{n}\subset \mathbb{C}^{n}$ be a  bounded product domain such that for each $1\leq i\leq n$ 
there is a singleton component $\{p_i\} $ of
 $\overline{\mathbb{C}}-\Omega_{i}$. Then for $z=(z_{1}, z_{2},\ldots, z_{n})\in \Omega$, 
 $T_{\Omega}(z) \leq \min \lbrace \sigma^{-1}(K_{\widetilde{\Omega}_{1}}(z_{1}, p_{1})), \sigma^{-1}(K_{\widetilde{\Omega}_{2}}(z_{2}, p_{2})),
 \ldots, \sigma^{-1}(K_{\widetilde{\Omega}_{n}}(z_{n}, p_{n})) \rbrace,$
where \\$\widetilde{\Omega}_{i}= \Omega_{i}\cup \lbrace p_{i} \rbrace $ and $K_{\widetilde{\Omega}_{i}}(\cdot, \cdot)$
 is the Kobayashi distance on $\widetilde{\Omega}_{i}$ for $i = 1, 2,\ldots, n$. In particular, $\Omega$ is not holomorphic homogeneous regular domain. 
\end{theorem}
\begin{proof}
For $z=(z_{1}, z_{2},\ldots, z_{n})\in \Omega$, let $T_{\Omega}(z) = c$. Then there exist a holomorphic embedding $f=(f_{1}, f_{2},\ldots, f_{n}) :\Omega\to \mathbb{D}^n$ 
such that $f(z)=0$ with $\mathbb{D}^n(0,c)\subseteq f(\Omega)$.
 So there is a holomorphic embedding $f_{i}:\Omega_{i}\to \mathbb{D}$ such that $f_{i}(z_{i})=0$ with $\mathbb{D}(0,c)\subseteq f_{i}(\Omega_{i})$ for each $i = 1, 2,\ldots, n$.

For each $1\leq i \leq n$ , by Riemann’s removable singularity theorem there exists a holomorphic function $\widetilde{f}_{i}: \widetilde{\Omega}_{i}\to \mathbb{D}$ 
with $\widetilde{f}_{i}(\zeta) = f_{i}(\zeta)$ for every $\zeta\in \Omega_{i}$. Now we will show that $\widetilde{f}_{i}(p_{i})\notin f_{i}(\Omega_{i})$. 
For this, let $\widetilde{f}_{i}(p_{i})\in f_{i}(\Omega_{i})$ and $\lbrace \zeta_{n} \rbrace$ be a sequence in $\Omega_{i}$ converging to $p_{i}$,
 then by continuity of $\widetilde{f}_{i}$ at $p_{i}$, $\lbrace \widetilde{f}_{i}(\zeta_{n})\rbrace$ converges to $\widetilde{f}_{i}(p_{i})$.
  This implies $\displaystyle\lim_{n\to \infty}f_{i}(\zeta_{n}) = \widetilde{f}_{i}(p_{i})$. By using continuity of $f_{i}^{-1}$,
$$f_{i}^{-1}(\widetilde{f}_{i}(p_{i})) = f_{i}^{-1}\left(\lim_{n\to \infty}f_{i}(\zeta_{n})\right) = \lim_{n\to \infty}\zeta_{n},$$
which is not possible. So
\begin{equation}
\widetilde{f}_{i}(p_{i})\notin f_{i}(\Omega_{i})
\end{equation}  
Also by decreasing property of Kobayashi distance 
\begin{align*}
K_{\mathbb{D}}(\widetilde{f}_{i}(z_{i}), \widetilde{f}_{i}(p_{i})) &\leq K_{\widetilde{\Omega}_{i}}(z_{i}, p_{i})\\
K_{\mathbb{D}}(0, \widetilde{f}_{i}(p_{i})) &\leq K_{\widetilde{\Omega}_{i}}(z_{i}, p_{i})\\
\sigma(|\widetilde{f}_{i}(p_{i})|) &\leq K_{\widetilde{\Omega}_{i}}(z_{i}, p_{i}).
\end{align*}
By using (2.1) $$c \leq |\widetilde{f}_{i}(p_{i})| \leq \sigma^{-1}(K_{\widetilde{\Omega}_{i}}(z_{i}, p_{i})).$$
Since $i$ is fixed, the result follows.
\end{proof}
\begin{theorem}
Let $\Omega = \Omega_{1}\times\Omega_{2}\subset \mathbb{C}^{2}$ be a bounded product domain such that $\overline{\mathbb{C}}-\Omega_{1}$ has finite number of components and none of them 
is singleton. Assume that $\Omega_{2}$ is a simply connected domain in $\mathbb{C}$. Then for $z=(z_{1}, z_{2})$
$$\lim_{z_{1}\to\partial\Omega_{1}}T_{\Omega}(z) = 1.$$
\end{theorem}
\begin{proof}
Let $E_{1}, E_{2},\ldots, E_{n},\ n\geq 1$ bee connected components of $\overline{\mathbb{C}}-\Omega_{1}$. Let $\Omega_{1}'= \overline{\mathbb{C}}\setminus E_{1}$ 
which is clearly a simply connected domain with $\Omega_{1}\subseteq\Omega_{1}'$. Since $E_{1}$ is not a singleton, by Riemann mapping theorem there is a injective 
holomorphic function from $\Omega_{1}'$ onto $\mathbb{D}$. Note that $f_{1}(\Omega_{1})$ is domain in $\mathbb{D}$ which we get after deleting $n-1$ connected 
compact subsets, say $E_{2}',\ldots, E_{n}'$.

For $z=(z_{1}, z_{2})\in \Omega$, let $l_{f_{1}(z_{1})}$ denotes the distance between $f_{1}(z_{1})$ and $\cup_{i=2}^{n}E_{i}'$ 
with respect to Poincaré metric. Let $\mathbb{P}(f_{1}(z_{1}), l_{f_{1}(z_{1})})$ be the disc centered at $f_{1}(z_{1})$ with radius $l_{f_{1}(z_{1})}$.
 Choose an injective  holomorphic function $g_{1}:\mathbb{D}\to\mathbb{D}$ with $g_{1}(f_{1}(z_{1})) = 0$. 
 Since Poincaré metric is invariant under conformal self map of $\mathbb{D}$, $g_{1}$ maps $\mathbb{P}(f_{1}(z_{1}), l_{f_{1}(z_{1})})$ 
 onto another disc centered at origin with same radius $l_{f_{1}(z_{1})}$ with respect to Poincaré metric. 

Also since $\Omega_{2}$ is simply connected bounded domain, there is an injective holomorphic function $f_{2}$ from $\Omega_{2}$ onto $\mathbb{D}$ with $f_{2}(z_{2}) = 0$.
 Clearly the disc centered at $f_{2}(z_{2})$ with radius $l_{f_{1}(z_{1})}$ with respect to Poincaré metric is contained in $f_{2}(\Omega_{2})$.
  Let $g_{2}$ be identity function on $\mathbb{D}$. Choose $g\circ f=(g_{1}\circ f_{1}, g_{2}\circ f_{2})$ on $\Omega$, then $\mathbb{D}^{2}(0, \sigma^{-1}(l_{f_{1}(z_{1})}))$
    is contained in $g(f(\Omega))$. For $f = (f_{1}, f_{2})$, $T_{f(\Omega)}(f(z))\geq \sigma^{-1}(l_{f_{1}(z_{1})})$.
    Since $|f_{1}(z_{1})|\to 1$, $l_{f_{1}(z_{1})}\to \infty$ therefore $\sigma^{-1}(l_{f_{1}(z_{1})})\to 1$. By invariant property of squeezing function 
$$\lim_{z_{1}\to E_{1}} T_{\Omega}(z)=1.$$
Similarly 
$$\lim_{z_{1}\to E_{i}} T_{\Omega}(z)=1,$$
for each $i = 2, 3,\ldots, n$. Hence
$$\lim_{z_{1}\to\partial\Omega_{1}}T_{\Omega}(z) = 1.$$ 
\end{proof}

\section{Examples}

\begin{example}
Let $\mathbb{D}\setminus\lbrace 0 \rbrace = \lbrace \zeta\in\mathbb{C} : 0 <|\zeta|<1 \rbrace$ be punctured disc in $\mathbb{C}$. Then for domain $\Omega =
 \mathbb{D}\setminus\lbrace 0 \rbrace\times\mathbb{D}\setminus\lbrace 0 \rbrace\times\ldots\times\mathbb{D}\setminus\lbrace 0 \rbrace\subset \mathbb{C}^{n}$ 
 and $z=(z_{1}, z_{2},\ldots, z_{n})\in \Omega$, 
$$T_{\Omega}(z) = \min \lbrace |z_{1}|, |z_{2}|,\ldots, |z_{n}| \rbrace.$$
For $z=(z_{1}, z_{2},\ldots, z_{n})\in \Omega$ by \cite[Proposition~4.6]{polydisc} 	
\begin{align*}
T_{\Omega}(z) &\geq \min \lbrace T_{\mathbb{D}\setminus\lbrace 0 \rbrace}(z_{1}), T_{\mathbb{D}\setminus\lbrace 0 \rbrace}(z_{2}),
\ldots, T_{\mathbb{D}\setminus\lbrace 0 \rbrace}(z_{n}) \rbrace\\
&= \min \lbrace S_{\mathbb{D}\setminus\lbrace 0 \rbrace}(z_{1}), S_{\mathbb{D}\setminus\lbrace 0 \rbrace}(z_{2}),\ldots, S_{\mathbb{D}\setminus\lbrace 0 \rbrace}(z_{n}) \rbrace\\
&= \min \lbrace |z_{1}|, |z_{2}|,\ldots, |z_{n}| \rbrace.
\end{align*}
\smallskip
By Theorem 2.1
$$T_{\Omega}(z) \leq \min \lbrace |z_{1}|, |z_{2}|,\ldots, |z_{n}| \rbrace.$$
\end{example}
\begin{example}
For domain $\Omega = \mathbb{D}\times\mathbb{D}\setminus\lbrace 0 \rbrace\subset \mathbb{C}^{2}$ and $z=(z_{1}, z_{2})\in \Omega$, 
$$T_{\Omega}(z) =  |z_{2}|.$$
For $z=(z_{1}, z_{2})\in \Omega$ by \cite[Proposition~4.6]{polydisc} 	
\begin{align*}
T_{\Omega}(z) &\geq \min \lbrace T_{\mathbb{D}}(z_{1}), T_{\mathbb{D}\setminus\lbrace 0 \rbrace}(z_{2})\rbrace\\
&= \min \lbrace S_{\mathbb{D}}(z_{1}), S_{\mathbb{D}\setminus\lbrace 0 \rbrace}(z_{2})\rbrace\\
&= |z_{2}|.
\end{align*}
\smallskip
By similar process as in Theorem 2.1
$$T_{\Omega}(z) \leq  |z_{2}|.$$
\end{example}
\begin{example} 
Let $A_{r} = \lbrace \zeta\in\mathbb{C} : r <|\zeta|<1, r>0\rbrace$ be annulus in $\mathbb{C}$.
 Then for domain $\Omega = A_{r}\times\mathbb{D}\subset \mathbb{C}^{2}$ and $z=(z_{1}, z_{2})\in \Omega$, 
$$T_{\Omega}(z)=
\begin{cases}
\dfrac{r}{|z_{1}|} &         \text{if}, r<|z_{1}|\leq\sqrt{r}\\
|z_{1}| &         \text{if},  \sqrt{r}<|z_{1}|<1.
\end{cases}$$
For $z=(z_{1}, z_{2})\in \Omega$ by \cite[Proposition~4.6]{polydisc} 	
\begin{align*}
T_{\Omega}(z) &\geq \min \lbrace T_{A_{r}}(z_{1}), T_{\mathbb{D}}(z_{2})\rbrace\\
&= \min \lbrace S_{A_{r}}(z_{1}), S_{\mathbb{D}}(z_{2}) \rbrace\\
&= \begin{cases}
\dfrac{r}{|z_{1}|} &   \text{if}, r<|z_{1}|\leq\sqrt{r}\\
|z_{1}| &         \text{if},  \sqrt{r}<|z_{1}|<1
\end{cases}.
\end{align*}
\smallskip
Let $T_{\Omega}(z) = c$. Then there exist a holomorphic embedding $f=(f_{1}, f_{2}) :\Omega\to \mathbb{D}^2$ such that $f(z)=0$ with $\mathbb{D}^2(0,c)\subseteq f(\Omega)$.
 So there is a holomorphic embedding $f_{1}:\Omega_{1}\to \mathbb{D}$ such that $f_{1}(z_{1})=0$ with $\mathbb{D}(0,c)\subseteq f_{1}(A_{r})$. 
 Since holomorphic embedding maps boundary to boundary, so either $f_{1}$ maps outer boundary of $A_{r}$ to outer boundary of $f_{1}(A_{r})$ or $f_{1}$ maps inner boundary of $A_{r}$ 
 to outer boundary of $f_{1}(A_{r})$. 

Assume $f_{1}$ maps outer boundary of $A_{r}$ to outer boundary of $f_{1}(A_{r})$. Let $\Omega'$ denotes the union of $f_{1}(A_{r})$ and compact connected component of
 $\mathbb{D}\setminus f_{1}(A_{r})$, then clearly $\Omega'$ is simply connected domain. 
 By Riemann mapping theorem we can find a injective holomorphic function $g_{1}$ from $\mathbb{D}$ onto $\Omega'$ with $g_{1}(0) = 0$. 
 Consider $\widetilde{f_{1}} = g_{1}^{-1}\circ f_{1}$ which is a holomorphic embedding from $A_{r}$ to $\mathbb{D}$ with $\widetilde{f_{1}}(z_{1})= 0$ and $\widetilde{f_{1}}(\partial\mathbb{D})=\partial\mathbb{D}$ such that $g_{1}^{-1}(\mathbb{D}(0, c))\subseteq \widetilde{f_{1}}(A_{r})$. Also by Schwarz lemma $g(\mathbb{D}(0, c))\subseteq\mathbb{D}(0, c)$. This implies $\mathbb{D}(0, c))\subseteq\widetilde{f_{1}}(A_{r})$. Then by \cite[Theorem~2]{recent}, $c \leq |z_{1}|$.

Assume $f_{1}$ maps inner boundary of $A_{r}$ to outer boundary of $f_{1}(A_{r})$. Let $g_{1}(\zeta)=\dfrac{r}{\zeta}$ be reflection on $A_{r}$. Consider $\widetilde{f_{1}} = f_{1}\circ g_{1}$ which is a holomorphic embedding from $A_{r}$ to $\mathbb{D}$ with $\widetilde{f_{1}}\left(\dfrac{r}{z_{1}}\right)= 0$ and $\widetilde{f_{1}}(\partial\mathbb{D})=\partial\mathbb{D}$ such that $\mathbb{D}(0, c))\subseteq\widetilde{f_{1}}(A_{r})$. Again by \cite[Theorem~2]{recent}, $c \leq r/|z_{1}|$.
\end{example}
\smallskip
\begin{remark} 
Domain $\mathbb{D}\setminus\lbrace 0 \rbrace\times\mathbb{D}\setminus\lbrace 0 \rbrace\times\ldots\times\mathbb{D}\setminus\lbrace 0 \rbrace\subset \mathbb{C}^{n}$ is example of a Pseudoconvex domain in higher dimensions which is not HHR. Also, if $\Omega = \Omega_{1}\times\Omega_{2}\times\ldots\times\Omega_{n}\subset \mathbb{C}^{n}$ be a domain such that for some $i = 1, 2, 3,\ldots, n$, $\Omega_{i}$ is punctured somewhere, then $\Omega$ is not HHR.      
\end{remark}

\section*{Ackowledgement}
We are thankful to our research supervisor Sanjay Kumar Pant for reading its previous draft and suggesting many changes.


\medskip

 
\begin{thebibliography}{00}

\bibitem{alexander} H. Alexander,  Extremal holomorphic imbeddings between the ball and polydisc, {\em Proc. Amer. Math. Soc.}, {\bf 68}(2) (1978), 200--202.

\bibitem{kobayashi-usual} T. J. Barth, The Kobayashi distance induces the standard topology, {\em Proc. Amer. Math. Soc.}, {\bf 35}(2) (1972), 439–441.
 
\bibitem{2012} F. Deng, Q. Guan, L. Zhang, Some properties of squeezing functions on bounded domains, {\em Pacific Journal of Mathematics}, {\bf 57}(2) (2012), 319--342.

\bibitem{2016} F. Deng, Q. Guan, L. Zhang; Properties of squeezing functions and global transformations of bounded domains, {\em Trans. Amer. Math. Soc.},  {\bf 368} (2016), 2679--2696.


\bibitem{deng2019} F. Deng, X. Zhang, Fridman's invariants, squeezing functions and exhausting domains, {\em Acta Math. Sin. (Engl. Ser.)}, {\bf 35}(2019), 1723--1728.

\bibitem{talk} J. E. Forn\ae ss, The squeezing function, Talk at Bulgaria academy of science national 
mathematics colloqium, 2019. 

\bibitem{frid1979} B. L. Fridman, On the imbedding of a strictly pseudoconvex domain in a polyhedron, {\em Dokl. Akad. Nauk SSSR}, {\bf 249}(1) (1979), 63–-67.
 
\bibitem{frid1983} B.L. Fridman, Biholomorphic invariants of a hyperbolic manifold and some applications, {\em Trans. Amer. Math. Soc.}, {\bf 276} (1983), 685--698.

\bibitem{polydisc} N. Gupta, S. K. Pant, Squeezing function corresponding to polydisk, arXiv:2007.14363v2[math.cv].

\bibitem{d-balanced} N. Gupta, S. K. Pant, Squeezing function for d-balanced domains, arXiv:2103.00526v1[math.cv].

\bibitem{kim-joo} Seungro Joo, Kang-Tae Kim, On boundary points at which the squeezing function tends to one, {\em The Journal of Geometric Analysis}, (2016), 10.1007/s12220-017-9910-4

 
\bibitem{uniform-squeezing} K.T. Kim, L. Zhang, On the uniform squeezing property of bounded convex domains in $\mathbb{C}^n$, {\em Pacific J.Math.}, {\bf 282}(2) (2016), 341--358.
 
\bibitem{riemann-removable} Steven G. Krantz., Function theory of several complex variables, AMS Chelsea Publishing, Providence, Rhode Island, 1992.
 
\bibitem{kubota} Y. Kubota, An extremal problem on the classical Cartan domain II, {\em Kodai Math. J},  {\bf 5} (1982), 218--224.
 
\bibitem{Yau2004} K. Liu, X. Sun,  S. T. Yau, Canonical metrics on the moduli space of Riemann surfaces, I, {\em J. Differential Geom.}, {\bf 68}(3) (2004), 571--637.
 
 
\bibitem{Yau2005} K. Liu, X. Sun, and S.T. Yau, Canonical metrics on the moduli space of Riemann
surfaces, II, {\em J. Differential Geom.}, {\bf 69}(1) (2005), 163--216.
 
 
\bibitem{Rouche} N.G. Lloyd, Remarks on generalising Rouche’s theorem, {\em J. London Math. Soc. (2)}, {\bf 20} (1979), 259--272.
 
\bibitem{kaushalprachi} P. Mahajan, K. Verma, A comparison of two biholomorphic invariants, {\em Internat. J. Math.}, {\bf 30}(1) (2019), 195--212.
 
\bibitem{recent}  T.W. Ng, C.C. Tang,  J. Tsai, The squeezing function on doubly-connected domains via the Loewner differential equation, {\em Math. Ann.} (2020). https://doi.org/10.1007/s00208-020-02046-w.
 
\bibitem{andreev} N. Nikolov,  L. Andreev,  Boundary behavior of the squeezing functions of C-convex domains and plane domains, {\em Internat. J. Math.}, {\bf 28}(5)(2017), https://doi.org/10.1142/S0129167X17500318, 5 pp.

\bibitem{kaushal} N. Nikolov, K. Verma, On the squeezing function and Fridman invariants, {\em J.
Geom Anal.}, {\bf 30} (2019), 1218--1225. 

\bibitem{rong} F. Rong, S. Yang, On Fridman invariants and generalized squeezing functions, Preprint (personal communication).

\bibitem{rudin:p} W. Rudin, Function Theory in Polydiscs, Benjamin, New York, 1969 
Springer-Verlag, New York-Berlin, 1980. 

\bibitem{rudin} W. Rudin, Function theory in the unit ball of $\mathbb{C}^n$, Berlin-Heidelberg-New York: Springer, 1980.


\bibitem{yeung} S. K. Yeung, Geometry of domains with the uniform squeezing property,
{\em  Adv. Math.}, {\bf 221}(2) (2009), 547--569.

\bibitem{zimmer2} A. Zimmer, A gap theorem for the complex geometry of convex domains, {\em Trans. Amer. Math. Soc.}, {\bf 370} (2018), 7489--7509.


\end{thebibliography}
\end{document}